\documentclass[a4paper,10pt,reqno]{amsart} 
\usepackage{amssymb} %\varnothing
\usepackage[colorlinks=true,urlcolor=black,citecolor=black,linkcolor=black]{hyperref}% for printing

\usepackage{indentfirst}
\usepackage[all]{xy}
\numberwithin{equation}{section}

\newtheorem{thm}{Theorem}[section]
\newtheorem{lemma}[thm]{Lemma}
\newtheorem{prop}[thm]{Proposition}
\newtheorem{cor}[thm]{Corollary}
\newtheorem{definition}[thm]{Definition}

\newtheorem{rmk}[thm]{Remark}

\newtheorem{bigthm}{Theorem}

\newcommand{\ignor}[1]{}

\emergencystretch=\maxdimen
\date{\today}
\address{\hfill\upshape\today.}

\begin{document}

%\linenumbers %to make \usepackage{lineno} works!

\title[Profinite extensions of centralizers]
{Profinite extensions of centralizers and the profinite completion of limit groups}
%\title[Subgroups of profinite extensions of centralizers]
%{Subgroups of profinite extensions of centralizers}
\author{Pavel Zalesskii}
\address{Department of Mathematics, University of Brasilia, Brazil}
\email{pz@mat.unb.br \textrm{and} zapata@mat.unb.br}
%\email{pz@mat.unb.br}
\author{Theo Zapata}

\maketitle

\vspace{-7pt}

\begin{quote}
\footnotesize
\textsc{Abstract.}
We introduce and investigate
a class of profinite groups 
defined via extensions of centralizers
analogous to the
extensively studied class of
finitely generated fully residually free groups,
that is, limit groups (in the sense of Z. Sela).
From the fact that
the profinite completion of limit groups
belong to this class, 
%stronger 
results 
on their group-theoretical structure
and homological properties
are obtained. 
\end{quote}

\vspace{21pt}
%\tableofcontents

\baselineskip 16.0pt

%%%%%%%%%%%%%%%%%%%%%%%%%%%%%%
%\newpage
\phantomsection
\section{Introduction}
\label{s:intro}
%%%%%%%%%%%%%%%%%%%%%%%%%%%%%%

In the last decades,
limit groups have been attracting the attention of many mathematicians
({\it e.g.}, \cite{CG:05}, \cite{AB:06}, \cite{BH:07}, \cite{CZ:07}, \cite{Wilton:08}, 
\cite{Kochloukova:10}, \cite{BK:17}),
specially because they played a fundamental role
in the recent solution 
(\cite{Sela:01} {\it et seq.} and \cite{KhM:98a} {\it et seq.})
of the outstanding Tarski problem on free groups.

Two basic group-theoretical reasons for so much attention are
the following:
limit groups are approximable by free groups,
according to their many definitions 
({\it cf.} \cite{Sela:01}, \cite{BBaumslag:67},  \cite{Remeslennikov:89}, \cite{KhM:98a}); 
and,
limit groups constitute a class of groups that envelopes
free groups of finite rank
and free abelian groups of finite rank.
Limit groups also have many interesting properties.
Some of their most simple properties are:
being torsion-free, 
having infinite abelianization (if the group is non-trivial),
that solubility implies commutativity,
being centerless, 
and being commutativity-transitive
({\it cf.} \cite[Prop.~1]{CG:05}).

Our present work introduces and studies
%(in \hyperref[s:class_Z]{Section~\ref{s:class_Z}})
a class of profinite groups $\mathcal Z$,
defined via extensions of centralizers
analogous to limit groups. 

A profinite extension of a centralizer is 
a free profinite amalgamated product $G\amalg_C A$, 
where $A$ is a finitely generated free abelian profinite group, 
$C$ is a procyclic direct factor of $A$, 
and $N_G(H)=C$ for every non-trivial closed subgroup $H$ of $C$.  
The class $\mathcal Z$ consists of all finitely generated closed subgroups 
of profinite groups obtained by finitely many iterated extensions of 
centralizers starting from a free profinite group of finite rank 
(see \hyperref[s:main_results]{Section~\ref{s:class_Z}} for details).

The following results are proved in 
%\hyperref[s:main_results]{Section~\ref{s:main_results}}.
sections \ref{s:class_Z} and \ref{s:main_results}.

%\newpage

\begin{bigthm}\label{t:cohom_main}
Let $G$ be a profinite group in the class $\mathcal Z$. Then
$G$ has finite cohomological dimension
and hence is torsion-free;
furthermore, 
if its cohomological $p$-dimension $cd_p(G)$ is $\ge 2$ % is not $p$-projetive 
then
$cd_p(G)=\max\{2,\alpha_p(G)\}$, where
$\alpha_p(G)$ is the supremum of the $cd_p(A)$
with $A$ running over the abelian closed subgroups of $G$. In particular, this gives  a uniform (finite) bound for the rank of abelian subgroups of $G$.
%\item[{\rm (\ref{p:good})}]
%the group 
\end{bigthm}

The next theorem is a statement of a Tits alternative flavour.

\begin{bigthm}\label{t:virtsol_main}
Let $G$ be a profinite  group in the class $\mathcal Z$. Then
%\item[{\rm (\ref{t:virtsol})}]
%the following assertions are equivalent:
%\begin{itemize}
%\item[{\rm (i)}] 
$G$ has no non-abelian free pro-$p$ closed subgroup
for each prime $p$ 
if, and only if,
%\item[{\rm (ii)}] 
$G$ is abelian %torsion-free f.g.
or meta-procyclic projective.
%\end{itemize}
\end{bigthm}

We also completely describe the 
finitely generated pro-$p$ subgroups of groups from class $\mathcal Z$.

\begin{bigthm} \label{t:fg_pro-p_main} 
%Let $G$ be a pro-$p$  group from class $\mathcal Z$. 
%Then $G=\coprod_i A_i$ is a free pro-$p$ product of
%%of free pro-$p$ groups or 
%free abelian pro-$p$ groups.
A pro-$p$ group in the class $\mathcal Z$ is 
a free pro-$p$ product of free abelian pro-$p$ groups.
\end{bigthm}

One of the main properties of limit groups is commutative transitivity. This properties already fails for free profinite groups. The next theorem describes those that are commutativity-transitive, showing that this property for profinite groups is quite restrictive.

\begin{bigthm}\label{t:CT_main}
A commutativity-transitive profinite  group in the class $\mathcal Z$
%Let $G$ be  a commutativity-transitive profinite  group from class $\mathcal Z$. Then $G$ 
is either pro-$p$ \textup{(}as described in Theorem \ref{t:fg_pro-p_main}\textup{)}
or abelian.
\end{bigthm}

We prove  that the profinite completion of a limit group belongs to the class $\mathcal Z$; that gives a source of important examples.  Thus all above theorems are true for finitely generated subgroups of the profinite completion of a limit group.
%\look{???} in particular Theorem \ref{t:fg_pro-p_main} gives a description of finitely generated pro-$p$ subgroups of the profinite completion of a limit group. 

\begin{bigthm} \label{t:pclg_main}
The profinite completion of a limit group belongs to the class $\mathcal Z$.
\end{bigthm}

We also establish more specific properties for the profinite completion of limit groups:

\begin{bigthm} \label{t:homological_main}
Let $G$ be the profinite completion of a limit group. Then:
%\item[{\rm (\ref{p:good})}] 
%the following cohomological properties hold:
%the group 
\begin{itemize}
\item[{\rm (a)}] 
$G$ is of homological type $\operatorname{FP}_\infty$ over $\mathbb{F}_p$ and over $\mathbb{Z}_p$.
\item[{\rm (b)}] 
the homological Euler-Poincar\'e characteristic of $G$ over $\mathbb{F}_p$ or $\mathbb{Z}_p$ 
coincides with the 
homological Euler-Poincar\'e characteristic $\chi(\Gamma)$ of $\Gamma$ over $\mathbb{Z}$; 
hence it is always $\leq 0$
and equals $0$ if, and only if, $\Gamma$ is abelian.
\item[{\rm (c)}]
if $U_1\supseteq U_2\supseteq U_3\supseteq \ldots$ is a 
descending sequence of open normal subgroups of $G$ 
with trivial intersection,
then
\[
\lim_{k\to\infty} 
\frac{\dim_{\mathbb{F}_p} {\rm H}_q(U_k,\mathbb{F}_p)}{(G:U_k)}
=
\lim_{k\to\infty} 
\frac{\dim_{\mathbb{F}_p} {\rm H}^q(U_k,\mathbb{F}_p)}{(G:U_k)}
=
\begin{cases} 
-\chi(\Gamma), & \textup{if } q=1, \\
\ \ \ \ 0,& \textup{otherwise} .
\end{cases}
\]
\end{itemize}
\end{bigthm}

The techniques involved in proving our 
%\hyperref[t:main_intro]{Theorem \ref{t:main_intro}}
theorems
consist of 
%, fundamentally,
the profinite version of Bass-Serre theory of groups acting on trees
({\it e.g.,} \cite{R:17}, \cite{ZM:88} and \cite{ZM:89})
%(groups acting on trees).
and homological methods
({\it e.g.,} \cite{Serre:94} and  
\cite{RZ:00b}).
%On the other hand,
%\hyperref[t:completion_intro]{Theorem \ref{t:completion_intro}}
The fact that the profinite completion of a limit group
belongs to our introduced class,
relies on the result established by {H. Wilton} 
\cite[Thm.~B]{Wilton:08} 
that limit groups virtually retract onto
each of its finitely generated subgroups,
and also on a result about centralizers
in the profinite completion of limit groups
obtained by {S. Chagas} and {P. Zalesskii}
\cite[Lemma~3.5]{CZ:07}.
Some facts of the 
combinatorial theory of profinite groups
are collected in 
%the 
\hyperref[s:prelim]{Section~\ref{s:prelim}}.
%of the paper.

Theorems \ref{t:cohom_main}, \ref{t:virtsol_main}, \ref{t:CT_main}, 
and \ref{t:pclg_main} 
were proved for our new class of profinite groups
in the Ph.D thesis \cite[Cap. 2]{tese}.

In the final section of the paper we comment
on group-theoretical properties of groups
belonging to the class of limit groups, 
to the class of pro-$p$ groups studied 
%in
by Kochloukova and Zalesskii
\cite{KZ:11},
and to the presently defined class of profinite
groups via extensions of centralizers.

\medskip

\noindent {\bf Notation.}
Throughout this paper, $p$ is a prime number.
The field of $p$ elements is denoted by
${\mathbb F}_p$.
The additive group of the ring of $p$-adic integers is ${\mathbb Z}_p$.
%;
%the natural numbers, ${\mathbb N}$.
%For $x$, $y$ in a group we shall write $y^x \!:= x^{-1} y x$.
If $\pi$ is a set of prime numbers, then ${\mathbb Z}_\pi$ denotes the
cartesian product $\prod_{p\in \pi}{\mathbb Z}_p$;
the set of all prime numbers but $p$
is denoted by $p'$.

For a subset $A$ of a topological group $G$ 
we denote by $\langle A\rangle$ 
the (abstract) subgroup of $G$ generated by $A$
and by $\overline{A}$ the topological closure of $A$ in $G$;
the subgroup of $G$ topologically generated by $A$ is
$\overline{\langle A\rangle}$.
%and by $A^G$ the normal closure of $A$ in $G$, {\it i.e.},
%the smallest closed normal subgroup of $G$ containing $A$.
By $d(G)$ we denote the smallest cardinality of a
(topological) generating subset of a profinite group $G$.
The profinite completion of a discrete group ${\Gamma}$ is $\widehat{\Gamma}$.

For a profinite group $G$ acting continuously
on a space $X$ 
we denote 
%the set of fixed points of $G$ by $X^G$ and its 
the {\em point stabilizer} 
for each $x$ in $X$
%is denoted
by $\operatorname{stab}_G(x)$.
%We define $\widetilde{G}:=\langle{ \operatorname{stab}_G(x) \,|\, x\in X}\rangle$.

Henceforth,
we use capital greek letters for discrete groups
and capital roman letters for profinite groups.
The rest of our notation is very standard and basically follows
\cite{RZ:00b}
and
\cite{RZ:00a}.
%\cite{RZ:00a} and \cite{RZ:00b}

%\medskip

%%%%%%%%%%%%%%%%%%%%%%%%%%%%%%
%\newpage
%\clearpage 
\phantomsection
\section{Preliminary results}
\label{s:prelim}
%%%%%%%%%%%%%%%%%%%%%%%%%%%%%%

In this section we collect a few 
results on the combinatorial theory of profinite groups
%properties of amalgamated free profinite products and
%profinite groups acting on profinite trees to be 
used in the paper.
Further information on this subject can be found in 
\cite{RZ:00b}, \cite{RZ:00a},
and
\cite{R:17}.
%\cite{RZ:00a} and \cite{RZ:00b}
%Finally,
For the cohomology theory of profinite groups
we refer to %the classical text of Serre
\cite[Ch. I]{Serre:94} and 
\cite[Ch. 6 and 7]{RZ:00b}.

%\newpage
An amalgamated free profinite product %$G\!:=
%$A\amalg_C B$ 
$G=G_1\amalg_{C} G_2$ is {\em proper}
if %$C$ is a proper subgroup of both $A$ and $B$.
the canonical maps $G_1\to G$ and $G_2\to G$ are injective.
%Unless differently stated we shall consider exclusively
%proper free profinite amalgamated products
%and 
We shall make use of the fact proved by 
{L.~Ribes}
\cite[Thm.~2.3]{Ribes:71}
({\it cf.} \cite[Exerc. 9.2.6]{RZ:00b})
that
\emph{an amalgamated free profinite product
%with amalgamating subgroup
is proper, whenever the amalgamating subgroup is
central in one of the amalgamated free factors}.

%\newpage
\begin{prop}[\protect{\cite[Cor. 2.7(ii)]{RZ:96}}]
\label{prop:eqnorm-rz96} %COROLLARY 2.7. (ii). 
Let $G_1$ and $G_2$ be profinite groups
with a common procyclic subgroup
$C$ such that $G=G_1\amalg_{C} G_2$ is proper.
Let $c\in C$.
The following relation among normalizers holds,
$N_{G}(\overline{\langle c\rangle}) = 
N_{G_1}(\overline{\langle c\rangle})\amalg_{C} 
N_{G_2}(\overline{\langle c\rangle})$.
\end{prop}

%\newpage

We assume basic concepts and facts of the profinite version 
of Bass-Serre theory as developed by Mel$'$nikov, Ribes, and Zalesskii
\cite{ZM:88}, \cite{ZM:89}, and \cite{R:17}; 
in particular, we only mention here that a profinite tree is
an acyclic connected profinte graph. 
%[which can have loops and 
%%two vertices may be connected by more than one edge
%multiple edges]. 
%%(``with own identity'').
%%https://en.wikipedia.org/wiki/Multigraph
For our purposes we emphasize the following two results.
The first is a sort of Tits alternative;
the second has no abstract analogue (here the compactness of the tree is essential).

\begin{thm}[\protect{\cite[Thm. 3.1]{Zalesskii:90}}]
\label{teo:caract-zal90} %THEOREM 3.1
Let $G$ be a profinite group acting on a profinite tree $T$.
Then one of the following holds: 
%\pagebreak
\begin{itemize}
\item[{\rm (a)}] 
$G$ stabilizes some vertex of $T$. 
\item[{\rm (b)}] 
$G$ has a non-abelian free pro-$p$ subgroup $P$
such that $P\cap\operatorname{stab}_G(v)=\{1\}$,
for each vertex $v$ of $T$.
\item[{\rm (c)}] 
There exists an edge $e$ of $T$, 
whose stabilizer $\operatorname{stab}_G(e)$ is normal in $G$
and $G/\operatorname{stab}_G(e)$ is isomorphic to
one of the following profinite groups: 
{\rm (c.1)} a projective group
$\mathbb{Z}_\pi\rtimes \mathbb{Z}_\rho$,
where $\pi \cap\rho=\varnothing$; 
{\rm (c.2)} a Frobenius group $\mathbb{Z}_\pi\rtimes\mathbb{Z}/m\mathbb{Z}$, 
where $m$ is not divisible by any prime of $\pi$
and $[k,c]\ne1$ for each $k\in\mathbb{Z}_{\pi}\!-\!\{0\}$
and each $c\in\mathbb{Z}/m\mathbb{Z}\!-\!\{0\}$;
{\rm (c.3)} the infinite dihedral pro-$\pi$ group
$\mathbb{Z}/ 2\mathbb{Z}\amalg^\pi \mathbb{Z}/2\mathbb{Z}$
isomorphic to
$\mathbb{Z}_\pi \rtimes \mathbb{Z}/2\mathbb{Z}$, 
where $2\in\pi$.
\end{itemize}
\end{thm}

\begin{prop}[\protect{\cite[Lemma 1.5]{Zalesskii:90} 
complemented by \cite[Prop.~19]{Serre:77}}] \label{prop:invar-zal90} %LEMMA 1.5
Let~$G$ be a profinite group acting on a profinite tree $T$.
There exists a \textup{(}non-empty\textup{)} minimal $G$-invariant profinite subtree
$D$ of $T$.
Moreover, if $D$ has more than one element, then $D$ is unique and infinite.
\end{prop}

%\newpage
In order to prove the main result of the following section,
we recall the concept of the 
\emph{Tits straight path}
(\textit{cf.} \cite[Prop.~3.2, p.~192]{Tits:70} %p.192
or 
\cite[Prop.~24]{Serre:77}). %p.63
Let $X$ be an abstract tree
and denote its (canonical) length function by $l$
and its vertex set by $V(X)$.
Let $\alpha$ be a fixed-point-free automorphism of $X$.
The set of all vertices $v$ of $X$ such that
$l(v,\alpha(v))=\inf_{w\in V(X)} l(w,\alpha(w))$
constitute the vertex set of a unique straight path
(\textit{i.e.}, a doubly infinite chain) $T_\alpha$;
this subtree $T_\alpha$ is called the
{Tits straight path} of $\alpha$.
If $[v,w]$ denotes the unique path joining vertices $v$ and $w$ in a tree
and $[v,w[\,=[v,w]-\{w\}$, 
then ${T_\alpha}={\langle \alpha\rangle}[v,\alpha v[\,$
for any $v$ in $T_g$.

For the reader's convenience the next proposition collects fundamental facts 
in the profinite usage of the Tits straight path
(\textit{cf.} \cite[proofs of Lemma~2.8 and Prop.~2.9]{RZ:96};
\cite[Example~1.20A]{ZM:88} and the previous
\hyperref[prop:invar-zal90]{Proposition~\ref{prop:invar-zal90}};
and \cite[proof of Lemma~4.3]{RSZ:98}).
%being items (a) and (b) the profinite analogue of the abstract
%case (\textit{cf.} \cite[proof of Prop.~24]{Serre:77}).

\begin{prop}%[\protect{\cite[Thm. 3.1]{Zalesskii:90}}]
\label{prop:profinite_tits}
Let $\Gamma= A *_Z B$ be a free product with 1-generated amalgamation.
Assume $\Gamma$ is residually finite, $\Gamma$ induces the profinite topology on $A$, $B$, and $Z$, and $A$, $B$, and $Z$ are closed in the profinite topology of $\Gamma$.
Let $\gamma$ be an element of $\Gamma$ not belonging to a conjugate of either $A$ or $B$
\textup{(}in other words, $\gamma$ is hyperbolic\textup{)}.
Let $T_\gamma$ be the Tits straight path of $\gamma$ in 
the standard tree $S$ on which $\Gamma$ acts,
and let $\overline{T_\gamma}$ be its closure in 
the standard profinite tree $\widehat{S}$ on which $\widehat{\Gamma}$ acts.
Then:
\begin{itemize}
\item[{\rm (a)}] 
%${T_g}={\langle g\rangle}[v,gv[$
%and
$\overline{T_\gamma}=\overline{\langle \gamma\rangle}[v,\gamma v[\,$
for any $v$ in $T_\gamma$; 
\item[{\rm (b)}]
%${T_g}$ is the unique 
%minimal ${\langle g\rangle}$-invariant subtree
%of ${S}$
%and
$\overline{T_\gamma}$ is the unique 
minimal $\overline{\langle \gamma\rangle}$-invariant profinite subtree
of $\widehat{S}$; 
%Moreover, if each 1-generated subgroup of $\Gamma$ is 
%closed in the profinite topology of $\Gamma$, then \\
\item[{\rm (c)}] 
$T_\gamma$ is a connected component of $\overline{T_\gamma}$ 
considered as as abstract graph, in other words, 
the only vertices of $\overline{T_\gamma}$  that are at a finite distance 
from a vertex of $T_\gamma$ are those of $T_\gamma$.
\end{itemize}
\end{prop}

%%%%%%%%%%%%%%%%%%%%%%%%%%%%%%
%\newpage
%\vspace{-0.3cm}
%\clearpage 
\phantomsection
\section{Profinite extensions of centralizers}
\label{s:class_Z}
%\section{Definition of profinite analogues of limit groups}
%%%%%%%%%%%%%%%%%%%%%%%%%%%%%%

Some properties of limit groups immediately pass to their profinite completion.
For instance, say $\widehat{\Gamma}$ is the profinite completion of a limit group.
If $\widehat{\Gamma}$ is virtually soluble, then it must be abelian;
each open subgroup of $\widehat{\Gamma}$ is finitely presented (as a profinite group); 
the group $\widehat{\Gamma}$ has infinite continuous abelianization 
if ${\Gamma}$ is non-trivial.

The present work introduces %and studies
a class of profinite groups %, called $\mathcal Z$,
defined via extensions of centralizers
analogous to limit groups.
We provide some motivation to define such class.

An extension of centralizer of a discrete group $\Gamma$ 
is an amalgamated free product $\Gamma*_Z (Z\times B)$
where $Z$ is the centralizer of an element of $\Gamma$ and 
$B$ is a free abelian group of finite rank.

%%Before we begin proving our theorems,
%%we shall discuss now the main reason
%%for the conditions in item (ii)
%%on the definition of the groups $G_n$.
%In any group consider 
%an element that generates a free group of rank 1.
In the discrete case,
the centralizer of each non-trivial element in a free group is 
infinite-cyclic
%1-generated,
and, after performing an extension of centralizer,
the centralizer of the element becomes abelian.
In the profinite case,
the centralizer of each 
non-trivial element that generates $\widehat{\mathbb{Z}}$
%free profinite subgroup of rank $1$ 
in a free profinite group is meta-procyclic
and, after performing an extension of centralizer,
the centralizer becomes either meta-abelian,
or (non-trivial procyclic)-by-(infinite dihedral pro-$\pi$),
or contains a non-abelian free pro-$p$ subgroup.
See 
\hyperref[l:sylow_cyclic]{Lemma~\ref{l:sylow_cyclic}}
and
\hyperref[t:virtsol]{Theorem~\ref{t:virtsol}},
%and 
%\hyperref[teo:caract-zal90]{Theorem~\ref{teo:caract-zal90}},
and 
\hyperref[p:centraliz.]{Propostition~\ref{p:centraliz.}}.

%\newpage
Our construction is now introduced.
Let $\mathcal C$ be a non-empty class of finite groups
closed 
%under formation of 
with respect to taking
subgroups, quotients,
and finite direct products.
We inductively define the following classes of pro-$\mathcal C$ groups:

{\itshape
\begin{quote}
The class
$\mathcal{Z}_0(\mathcal{C})$ 
%is the class 
of all
free pro-$\mathcal C$ groups of finite rank.

\noindent
For each integer $n>0$, the class $\mathcal{Z}_n(\mathcal{C})$ %$(n>0)$ is the class 
of all
%pro-$\mathcal C$
groups $G_n$,
%where
such that $G_n$ is the 
amalgamated free pro-$\mathcal C$ product
$G_{n-1}\amalg_{C_{n-1}} A_{n-1}$ satisfying the following conditions:
\begin{itemize}
\item[{\bf A1.}] %\label{d:(i)}
$C_{n-1}$ is a procyclic proper 
%closed 
subgroup and a direct factor
of a free abelian pro-$\mathcal C$ group of any finite rank $A_{n-1}$. 
%\notetous{do we want $C_{n-1}$  to be free?}
\item[{\bf A2.}] \label{d:(ii)}
$C_{n-1}$ is a closed subgroup of the group 
$G_{n-1}$ of $\mathcal{Z}_{n-1}(\mathcal{C})$ such that
the normalizer in $G_{n-1}$ of each non-trivial procyclic subgroup of
$C_{n-1}$ coincides with $C_{n-1}$. 
\end{itemize}
\end{quote}
}

%\newpage
\begin{definition} \label{defi:plg}
The class $\mathcal{Z}(\mathcal{C})$ of pro-$\mathcal C$ groups 
consists of all finitely generated profinite subgroups of some
group $G_n$ of $\mathcal{Z}_n(\mathcal{C})$ with $n\geq 0$.
When $\mathcal C$ consists of all finite groups, we simply denote
$\mathcal{Z}_n(\mathcal{C})$ by $\mathcal{Z}_n$,
and 
$\mathcal{Z}(\mathcal{C})$ by $\mathcal{Z}$.
\end{definition}

Note that, 
when $\mathcal C$ consists of all finite $p$-groups
the class $\mathcal{Z}(\mathcal{C})$
coincides precisely with the class of groups originally studied 
by Kochloukova and Zalesskii \cite{KZ:11}.
Indeed,
the property of limit groups that
the normalizer and the centralizer
of any non-trivial 1-generated subgroup 
coincide and are abelian
holds for such pro-$p$ groups
(\cite[Thm.~5.1]{KZ:11}).

\begin{rmk} \label{rmk:abelian_normalizers}
{\rm
It follows from %item (ii)
{condition {\bf A2}}
in the definition of the class $\mathcal{Z}_n$,
%(p.~\pageref{d:(ii)})
that 
{\it each non-trivial procyclic closed subgroup $P$ of $C_{n-1}$
has abelian normalizer in $G_n$}.
In fact, in light of 
\hyperref[prop:eqnorm-rz96]{Proposition~\ref{prop:eqnorm-rz96}}
we have
\[ %\label{eq:normalizadorC}
N_{G_n}(P) = 
N_{G_{n-1}}(P) \amalg_{C_{n-1}} N_{A_{n-1}}(P) =
{C_{n-1}} \amalg_{C_{n-1}} {A_{n-1}} = A_{n-1}  \ .	
\]
}
\end{rmk}

%\newpage
The next theorem,
whose pro-$p$ analogue is unknown, gives the main source of important examples of groups from
the class $\mathcal Z$.
Before we prove it, 
recall that a group acts \emph{$k$-acylindrically}
on a graph if 
only the trivial element fixes pointwise
each subset of diameter $>k$.
Note that 
the action of $G_n$ 
on its standard profinite tree is not $1$-acylindrical, 
since
%the 
two distinct edges $1C_{n-1}$ and $gC_{n-1}$
incident to the vertex $1A_{n-1}$
have stabilizers whose intersection is $C_{n-1}$.

\begin{lemma} \label{l:2acyl}
For each $n\ge 1$, the action of $G_n$ 
on its standard profinite tree is $2$-acylindrical.
\end{lemma}

\begin{proof}
Fix $n$.
We shall prove that any 
three consecutive 
edges have stabilizers that intersect trivially.
After a translation in the profinite tree of $G_n$,
we obtain
%that adjacent edges become
{two} distinct edges $1C_{n-1}$ and $gC_{n-1}$
incident to the vertex $1G_{n-1}$;
here, 
$g\in \protect{G_{n-1}\!-\!C_{n-1}}$.
Note that
$gC_{n-1}g^{-1}$ is not contained in $C_{n-1}$
(indeed, 
$N_{G_{n-1}}(C_{n-1}) $ equals $
\{x\in G_{n-1} \mid xC_{n-1}x^{-1}\subseteq C_{n-1}\}$
because closed subsemigroups in compact groups are subgroups,
and $N_{G_{n-1}}(C_{n-1})=C_{n-1}$ by definition).
%{condition {\bf A2}}).
% hence $gC_{n-1}g^{-1}\neq C_{n-1}$.
%Let $P=C_{n-1}\cap gC_{n-1}g^{-1}$.
On the other hand,
since $C_{n-1}$ is abelian,
$gC_{n-1}g^{-1}$ normalizes 
%the intersection 
$C_{n-1}\cap gC_{n-1}g^{-1}$.
%because it centralizes $P$, .
%If that intersecion were non-trivial, 
%then its normalizer in $G_{n-1}$ would be $C_{n-1}$,
%by definition of $G_n$. %cond. (ii)
%Therefore that intersection is trivial and the result follows.
So, this intersection must be trivial;
otherwise,
its normalizer in $G_{n-1}$ would be
$C_{n-1}$ by %definition.
{condition {\bf A2}}.	
The result follows.
\end{proof}

%\newpage
\begin{thm} \phantomsection \label{t:completion}
The profinite completion of a limit group 
belongs to the class $\mathcal Z$.
\end{thm}

\begin{proof}
First, let $\Gamma$ be a finitely generated subgroup of
a free group $\Gamma_0$. %(discrete)
Then $\Gamma$ is a free group and its profinite completion
is a free profinite group of finite rank.

Next, let $\Gamma$ be a 
finitely generated subgroup of the limit group %(discrete)
$\Gamma_n$ given by
$\Gamma_n=\Gamma_{n-1} *_{Z_{n-1}} A_{n-1}$,
where
$Z_{n-1}$ is a 1-generated proper subgroup and a direct factor of 
a finite rank free abelian group $A_{n-1}$,
%a free abelian group $A_{n-1}$ of finite rank,
and $Z_{n-1}$ is self-centralizing in $\Gamma_{n-1}$.

Note that
$\widehat{\Gamma_n}=
\widehat{\Gamma_{n-1}}\amalg_{\widehat{Z_{n-1}}}\widehat{A_{n-1}}$.
Indeed, 
%since 
the group $\Gamma_n$ is residually free;
%, it is residually finite; 
moreover, 
by \cite[Lemma~2.1]{RZ:96},
the profinite topology of $\Gamma_n$ induces 
the profinite topology on $\Gamma_{n-1}$, $A_{n-1}$, and $Z_{n-1}$;
so, the stated equality follows %immediately from
(\textit{cf.} \cite[Ex.~9.2.7(2)]{RZ:00b}).

Now, 
%arguing as in the first paragraph, we have 
we prove that the group $\widehat{\Gamma}$
is naturally embedded in $\widehat{\Gamma_n}$.
Indeed, Wilton
\cite[Thm.~B]{Wilton:08}
showed that
%limit groups virtually retract onto
%each of its finitely generated subgroups; that is, 
there exists a finite index subgroup
$\Upsilon$ in $\Gamma_n$ containing $\Gamma$ and there exists a group homomorphism
$\rho\colon \Upsilon\to \Gamma$ 
%such that the restriction $\rho_{|\Gamma}$ equals $\hbox{id}_\Gamma$;
that is an extension of the identity map of $\Gamma$;
thus, $\Upsilon$ is isomorphic to
$\ker(\rho)\rtimes \Gamma$.
By \cite[Lemma~3.1.4(a) and Lemma~3.1.5(a)]{RZ:00b}, 
the profinite topology of $\Gamma_n$ induces 
the profinite topology on $\Upsilon$ and
%by \cite[Lemma.~3.1.5(a)]{RZ:00b}, 
the profinite topology of $\Upsilon$ induces
the profinite topology on $\Gamma$.
So, the desired embedding follows 
(\textit{cf.} \cite[Lemma~3.2.6]{RZ:00b}).

To conclude the proof, we shall establish the last requirement of %condition (ii).
{condition {\bf A2}}.

\smallskip

%\newpage
\noindent \textit{Claim.
Let $\Gamma$ be a limit group.
If $Z$ is a 1-generated subgroup of $\Gamma$
such that $N_{\Gamma}(Z)=Z$,
then 
$N_{\widehat{\Gamma}}(\overline{\langle g\rangle})
=\overline{Z}$
for each non-trivial element $g$ in $\overline{Z}$.}
\textup{(Here the closures are taken in $\widehat{\Gamma}$.)}

\smallskip

It suffices to prove the 
claim
for $\Gamma=\Gamma_n$.
%If $n\ge 1$, say $\Gamma_n=\Gamma_{n-1} *_{Z_{n-1}} A_{n-1}$,
%where
%$Z_{n-1}$ is a 1-generated proper subgroup and a direct factor of 
%a finite rank free abelian group $A_{n-1}$,
%and $Z_{n-1}$ is self-centralizing in $\Gamma_{n-1}$.
Let $S$ be the standard tree on which $\Gamma_n$ acts
(if $n=0$, then $S$ is the Cayley graph of $\Gamma_0$).
%By induction, assume the 
%claim 
%is true for $\Gamma=\Gamma_{n-1}$ with $n\ge 1$;
%the basis of induction is proved latter.
Let $\zeta$ be a generator of $Z$.

%\textit{Case 1.} 
Suppose $\zeta$ does not fix any vertex of $S$
(\textit{i.e.}, the element $\zeta$ is hyperbolic).
Fix $n\ge 1$.
In order to apply 
\hyperref[prop:profinite_tits]{Proposition~\ref{prop:profinite_tits}},
recall that Wilton \cite[Thm.~A]{Wilton:08}
proved that limit groups are subgroup separable.
%lerf (\textit{i.e.}, subgroup separable).
Let $T_\zeta$ be the Tits straight path of $\zeta$ in $S$ 
and let $\overline{T_\zeta}$ be its closure in 
the standard profinite tree $\widehat{S}$ on which $\widehat{\Gamma_n}$ acts.

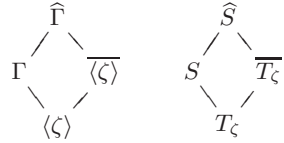
\begin{small}
\begin{equation*}
\xymatrix@R=5pt@C=0pt{
& \widehat{\Gamma} \ar@{-}[dl] \ar@{-}[dr] \\
\Gamma \ar@{-}[dr]
& & {\overline{\langle \zeta\rangle}} \ar@{-}[dl] \\ %G\!\!\!\!\!\!\! \phantom
& \langle \zeta\rangle \\
}
\qquad
\xymatrix@R=5pt@C=0pt{
& \widehat{S} \ar@{-}[dl] \ar@{-}[dr] \\
S \ar@{-}[dr]
& & \overline{T_\zeta} \ar@{-}[dl] \\ 
& T_\zeta \\
}
\end{equation*}
\end{small}

\noindent
\emph{Let $N=N_{\widehat{\Gamma}}(\overline{\langle g\rangle})$.}
By 
\hyperref[prop:profinite_tits]{Proposition~\ref{prop:profinite_tits}}(b),
%the normalizer
$N$ acts naturally on $\overline{T_\zeta}$;
for, 
the translation of $\overline{T_\zeta}$ by an element of $N$
is a minimal $\overline{\langle \zeta\rangle}$-invariant profinite subtree
of $\widehat{S}$.
Now, note that $N$
acts freely on $\overline{T_\zeta}$.
Indeed,
if $h$ in $N$ fixes
some element of $\overline{T_\zeta}$, then, by 
\hyperref[prop:profinite_tits]{Proposition~\ref{prop:profinite_tits}}(a),
there exists $c$ in $\overline{\langle g\rangle}$ such that
$c^{-1}hc$ fixes an element of $T_\zeta$;
so, by \hyperref[prop:profinite_tits]{Proposition~\ref{prop:profinite_tits}}(c),
$T_\zeta\cup c^{-1}hcT_\zeta=T_\zeta$;
whence, $c^{-1}hc$ acts on $T_\zeta$.
Let $H$ be the set of all elements of $N$ that leave $T_\zeta$ invariant;
from
the \mbox{2-acylindrical} action of $\widehat{\Gamma}$ on $\widehat{S}$     %?2-
(\hyperref[l:2acyl]{Lemma~\ref{l:2acyl}}),
the obvious map $H \to \textup{Aut}(T_\zeta)$ is an embedding;
so, being $H$ torsion-free
(\hyperref[t:cohom_main]{Theorem~\ref{t:cohom_main}})
and being $\textup{Aut}(T_\zeta)$ the free product of two groups of order 2
(\textit{i.e.,} the infinite dihedral group of $\mathbb Z$),
it follows that
the automorphism induced by $c^{-1}hc$ is the identity,
that is, $h=1$.

Next, we use an argument from covering theory.
By 
\hyperref[prop:profinite_tits]{Proposition~\ref{prop:profinite_tits}}(a),
$N\backslash\overline{T_\zeta}$ is 
a quotient of the finite circuit $\langle\zeta\rangle\backslash{T_\zeta}$;
on the other hand, 
the connected profinite graph $\overline{T_\zeta}$
is a simply connected profinite graph,
because so is $\widehat{S}$ 
(\textit{cf}. \cite[Lemma~A in~1.3 and Thm.~3.4]{ZM:89});
%\cite[Thm.~3.4]{ZM:89} says that $\widehat{S}$ is a simply connected profinite graph,
%and hence so is $\overline{T_\zeta}$;
therefore, by \cite[Thm.~2.8 and Lemma~2.3]{Zalesskii:89}
we get that
$N$ is isomorphic to the abelian group $\widehat{\mathbb Z}$.

So, we obtain
\[
\overline{Z}
\subseteq N
%_{\widehat{\Gamma}}(\overline{\langle g\rangle})
%
%\subseteq G
\subseteq C_{\widehat{\Gamma}}(\overline{Z})
= \overline{C_{\Gamma}(Z)} 
= \overline{N_{\Gamma}(Z)} \ ,
\]
where the last two equalities
are from
\cite[Prop.~3.5 and Lemma~3.3]{CZ:07}.

Note that
the proof for $n=0$ is, \textit{mutatis mutandis},
the same (\textit{cf}.  \cite[proof of Lemma~3.4]{RZ:96}).
This case is done.

%\textit{Case 2.} 
Finally, suppose $\zeta$ fixes some vertex of $S$.
Then, since 
%Z is maximal abelian
it cannot fix a vertex that is a coset of $A_{n-1}$,
it is, up to conjugation in $\Gamma_n$, an element in $\Gamma_{n-1}$. By \cite[Thm 3.12]{ZM:88} $N_{\widehat\Gamma}((\overline{\langle g\rangle})=N_{\widehat\Gamma_{n-1}}((\overline{\langle g\rangle})$ so that
the 
claim 
follows by induction on $n$.
\end{proof}

%\newpage
Simple examples of groups in the class $\mathcal{Z}$
are therefore the profinite completion of surface groups,
except those of non-orientable surfaces 
of genus 1, 2, or 3 (\textit{i.e.}, 
the real projective plane, the Klein bottle, and Dyck's surface).
%of Euler charateristic 1, 0, or -1.
%(why $G = \langle a,b,c|a^2b^2c^2 = 1\rangle\widehat{ \ } \ $?)

%\newpage

It is clear that the free profinite product of 
the profinite completion of two limit groups 
belongs to the class $\mathcal{Z}$.
The next result 
provides more examples of groups in the class $\mathcal{Z}$.

The important construction of {G. Baumslag}
known, nowadays, as {\it Baumslag double}
arises in his classical paper
\cite{Baumslag:62} on residually free groups.
It consists of an amalgamated free product such that
the amalgamated free factors are isomorphic groups and
the amalgamated subgroup is cyclic and 
self-normalizing in each amalgamated free factor.
Adding an obvious condition we obtain %, just like in the discrete and pro-$p$ situations, 
that a Baumslag double of a group in the class $\mathcal Z$
stays in the class $\mathcal Z$:

\begin{prop} \phantomsection \label{p:double}
Let $G$ be %a group in the class $\mathcal{Z}$
a finitely generated profinite group
contained in some group $G_n$ of $\mathcal{Z}_n$, 
let $\varphi\colon G\to {G}^{\varphi}$
be an isomorphism of profinite groups.
Assume that
$C$ is a procyclic subgroup of $G$ such that
for each non-trivial procyclic subgroup $D$ of $C$
we have $N_{G_n}(D)=C$.
%If $K$ is the amalgamated free profinite product
%% with procyclic amalgamation
Then the amalgamated free profinite product
$G\amalg_{C={C}^{\varphi}} {G}^{\varphi}$
belongs to $\mathcal{Z}_{n+1}$.
\end{prop}

\begin{proof}
Let
$C_n=C$ and $A_n=C\times \overline{\langle t\rangle}$
with $t$ such that $A_n$ is a free profinite abelian group of rank $2$,
and define
$G_{n+1}=G_n\amalg_{C_n} A_n$.
From hypothesis,
$G_{n+1}$ belongs to $\mathcal{Z}_{n+1}$.
Since $G\amalg_{C={C}^{\varphi}} {G}^{\varphi}$
is isomorphic to the subgroup
$\overline{\langle G,tGt^{-1}\rangle}$ of
$G\amalg _{C_n} A_n$,
%we have that $G\amalg_{C={C}^{\varphi}} {G}^{\varphi}$ 
it is a finitely generated closed subgroup of $G_{n+1}$.
\end{proof}

One may apply this proposition to $G=\widehat{\Gamma_n}$
and $C=\overline{Z}$ with $Z$ a 1-generated self-normalizing
subgroup of $\Gamma_n$
(\textit{cf.} Claim in the proof of 
\hyperref[t:completion]{Theorem~\ref{t:completion}}).
The proof of the proposition carries over in the discrete and pro-$p$ situations.

%%%%%%%%%%%%%%%%%%%%%%%%%%%%%%
%\newpage
%\clearpage 
\phantomsection
%\section{Cohomological properties}
\section{Main results}
\label{s:main_results}
%%%%%%%%%%%%%%%%%%%%%%%%%%%%%%

%\newpage
We begin this section by proving the homological Theorems 
\ref{t:cohom_main} and \ref{t:homological_main}.

\begin{proof}[Proof of 
\protect{\textup{\hyperref[t:cohom_main]{Theorem~\ref{t:cohom_main}}}}]
Say $G$ is a closed subgroup of some group
$G_n$ of $\mathcal{Z}_n$.
If $n=0$, we have $cd_p(G)\le 1$.
Suppose then $n\ge 1$.
%$G_n=G_{n-1}\amalg_{C_{n-1}} A_{n-1}$
In virtue of a Mayer-Vietoris sequence
associated to the amalgamated free product
$G_{n-1}\amalg_{C_{n-1}} A_{n-1}$
({\it cf.} \cite[(2.7)]{ZM:88}),
we obtain
\begin{equation*}
cd_p(G)\le
\sup\{2, cd_p(G\cap xG_{n-1}x^{-1}), cd_p(G\cap yA_{n-1}y^{-1})
\mid x\in G_n,y\in G_n\} \ ,
\end{equation*}
because $cd_p(C_{n-1})\le 1$.

Let
$\alpha_p(G)=
\sup\{ \ cd_p(A) \mid A\le_c G,$ $A \hbox{ abelian} \}$.
Then, 
%since
%$cd_p(G\cap xG_{n-1}x^{-1})$ is $\le 1$,
by induction %we obtain 
$cd_p(G\cap xG_{n-1}x^{-1})$
%it is
$\le \max\{2,\alpha_p(G)\}$;
and clearly 
$\protect{cd_p(G\cap yA_{n-1}y^{-1})}\le \alpha_p(G)$.
Hence 
$cd_p(G)\le \max\{2,\alpha_p(G)\}$.
Furthermore, if $cd_p(G)\ge 2$, % is not $p$-projetive
then we plainly have
$\max\{2,\alpha_p(G)\}\le cd_p(G)$.
\end{proof}

To prove 
\hyperref[t:homological_main]{Theorem~\ref{t:homological_main}}
we use the fact
established by Grunewald, Jaikin-Zapirain, and Zalesskii
[GJZ08, Thm.~1.3]
that limit groups are good (in the sense of Serre \cite[I.\S2.6]{Serre:94}).
A discrete group is 
%called \emph{good} 
good
if the Galois cohomology of its profinite completion 
coincides with the corresponding (discrete) group cohomology
for finite discrete modules.
More precisely:
a discrete group $\Gamma$ is 
%good
called \emph{good}
if for each finite discrete $\widehat{\Gamma}$-module $M$, the
obvious homomorphism
\[
{\rm H}^q(\widehat{\Gamma},M)
\longrightarrow {\rm H}^q(\Gamma,M)
\]
is bijective for any $q\ge 0$.

%Note that, 
Regarding \hyperref[t:cohom_main]{Theorem~\ref{t:cohom_main}}, 
note that
if $G$ is the profinite completion of a limit group $\Gamma$,
for each $q>cd(\Gamma)$ we have ${\rm H}^{q}(G,M)=0$ 
for every discrete $G$-module $M$ which is finite, 
and hence which is torsion.
So $cd(G)\leq cd(\Gamma)$.

\begin{lemma}\label{l:goody}
Let $\Gamma$ be a discrete group and 
let $M$ be a discrete finite $\widehat{\Gamma}$-module.
Consider the direct system \textup{(}resp., the inverse system\textup{)}
consisting of the 
restriction maps ${\rm H}^q(N,M)\to {\rm H}^q(N',M)$
\textup{(}resp., 
corestriction maps ${\rm H}_q(N',M)\to {\rm H}_q(N,M)$\textup{)},
over the indexing set of all finite index normal subgroups $N$ of $\Gamma$.
We have:

\begin{itemize}
\item[\rm (i)]
Goodness of $\Gamma$ implies $\varinjlim_N {\rm H}^q(N,M)=0$
for all $q\ge 1$.
\item[\rm (ii)]
If ${\rm H}_q(N,\mathbb{Z})$ is finitely generated 
for each sufficiently small $N$ 
%whenever $M$ is a trivial $N$-module, 
for all $q\ge 0$,
then 
$\varinjlim_N {\rm H}^q(N,M)=0$ implies $\varprojlim_N {\rm H}_q(N,M)=0$
for all $q\ge 1$.
\end{itemize}
\end{lemma}

\begin{proof} 
%\begin{small}
\noindent
\textrm{(i)} Fix $q$ and $N$. It suffices to prove that 
for each $x$ in ${\rm H}^q(N,M)$ 
there exists a finite index subgroup $N_x$ of $N$ such that
the image of $x$ 
under
%by 
the map ${\rm H}^q(N,M)\to {\rm H}^q(N_x,M)$ 
in the direct system
is zero.
%Hence $\varinjlim H^q(U,M)=0$.
This is a well-known exercise of Serre \cite[I.\S2.6, Exerc.~1, p.~13]{Serre:94}.

\textrm{(ii)}
We may and do assume that $M$ is a trivial $N$-module 
for each sufficiently small $N$;
by cofinality this does not alter the limits.
Fix $q\geq 1$ and $N$. 
%OLD PROOF
%We proceed by induction on the order of $M$.
%If it is one, there is nothing to do.
%Assume first that $M$ has prime order. %, \textit{i.e.}, $M=\mathbb{F}_p$. 
%By the Universal coefficient theorem %for cohomology 
%(\cite[Ch.~VI, Thm.~3.3a]{CE:56}),
%${\rm H}^q(N,\mathbb{F}_p)$ is isomorphic 
%to ${\rm Hom}({\rm H}_q(N,\mathbb{F}_p),\mathbb{F}_p)$
There exists a finite index subgroup $N'$ of $N$ such that
the map ${\rm H}^q(N,M)\to {\rm H}^q(N',M)$
in the direct system is the null map. 
%because dir.lim =0.
It suffices to prove that 
the corestriction ${\rm H}_q(N',M)\to {\rm H}_q(N,M)$ is the null map.
%\end{small}

By the Universal coefficient theorem
(\textit{cf.} \cite[Thm.~3.3 and Thm.~3.3a of Ch.~VI]{CE:56})
we have the following commutative diagrams with
splitting exact rows of abelian groups
\[\xymatrixcolsep{1pc}\xymatrixrowsep{1.5pc}
\xymatrix{
\{0\} \ar[r] & {\rm H}_{q}(N',\mathbb{Z}) \otimes M \ar[r] \ar[d] &
{\rm H}_q(N',M) \ar[r] \ar[d] & {\rm Tor}({\rm H}_{q-1}(N',\mathbb{Z}),M) \ar[r] \ar[d] & \{0\} \\
\{0\} \ar[r] & {\rm H}_{q}(N,\mathbb{Z}) \otimes M \ar[r] &
{\rm H}_q(N,M) \ar[r] & {\rm Tor}({\rm H}_{q-1}(N,\mathbb{Z}),M) \ar[r] & \{0\}  }
\]
and
\[\xymatrixcolsep{1pc}\xymatrixrowsep{1.5pc}
\xymatrix{
\{0\} \ar[r] & {\rm Ext}({\rm H}_{q-1}(N,\mathbb{Z}),M) \ar[r] \ar[d] &
{\rm H}^q(N,M) \ar[r] \ar[d] & {\rm Hom}({\rm H}_{q}(N,\mathbb{Z}),M) \ar[r] \ar[d] & \{0\} \\
\{0\} \ar[r] & {\rm Ext}({\rm H}_{q-1}(N',\mathbb{Z}),M) \ar[r]  &
{\rm H}^q(N',M) \ar[r]  & {\rm Hom}_{\mathbb Z}({\rm H}_{q}(N',\mathbb{Z}),M) \ar[r]  & \{0\} }
\]
where 
$\otimes=\otimes_{\mathbb Z}$, %and
${\rm Tor}={\rm Tor}_1^{\mathbb Z}$, 
${\rm Ext}={\rm Ext}^1_{\mathbb Z}$, and
${\rm Hom}={\rm Hom}_{\mathbb Z}$.

Let ``asterisk'' denote applying ${\rm Hom}_{\mathbb Z}(\,\textbf{--}\,,\mathbb{Q}/\mathbb{Z})$.
Duality homomorphisms 
(\textit{cf.} \cite[Prop.~5.4 of Ch.~VI and Prop.~5.2$'$ of Ch. II]{CE:56}
-- here the finiteness hypothesis on ${\rm H}_q(N,\mathbb{Z})$ is used)
together with $M^*\cong M$, 
give us isomorphisms as horizontal arrows
in the two subsequent diagrams
\[\xymatrixcolsep{1pc}\xymatrixrowsep{1.5pc}
\xymatrix{
{\rm Tor}({\rm H}_{q-1}(N',\mathbb{Z}),M) \ar[r] \ar[d] & {\rm Tor}(M^*,{\rm H}_{q-1}(N',\mathbb{Z})) \ar[r] \ar[d] &
{\rm Ext}({\rm H}_{q-1}(N',\mathbb{Z}),M)^* \ar[d] \\
{\rm Tor}({\rm H}_{q-1}(N,\mathbb{Z}),M) \ar[r] & {\rm Tor}(M^*,{\rm H}_{q-1}(N,\mathbb{Z})) \ar[r] &
{\rm Ext}({\rm H}_{q-1}(N,\mathbb{Z}),M)^* }
\]
and
\[\xymatrixcolsep{1pc}\xymatrixrowsep{1.5pc}
\xymatrix{
{\rm Hom}({\rm H}_{q}(N,\mathbb{Z}),M) \ar[r] \ar[d] & {\rm Hom}({\rm H}_{q}(N,\mathbb{Z}),M^*) \ar[r] \ar[d] &
({\rm H}_{q}(N,\mathbb{Z}) \otimes M)^* \ar[d] \\
{\rm Hom}({\rm H}_{q}(N',\mathbb{Z}),M) \ar[r] & {\rm Hom}({\rm H}_{q}(N',\mathbb{Z}),M^*) \ar[r] &
({\rm H}_{q}(N',\mathbb{Z}) \otimes M)^*  \ .
 }
\]
Since the restriction is null, it follows that so is the corestriction.
\end{proof}

We mention that both reverse implications of the previous Lemma hold.
The finiteness hypothesis in item (ii) is satisfied for groups
of type FP over $\mathbb{Z}$.

%\newpage
\begin{proof}[Proof of 
\protect{\textup{\hyperref[t:cohom_main]{Theorem~\ref{t:homological_main}}}}]
%Let $G=\widehat{\Gamma}$ for some limit group $\Gamma$.
Let $G$ be the profinite completion $\widehat{\Gamma}$ of a limit group $\Gamma$.

\noindent
(a)
%[ref. abstrato para FP$_\infty$ over $\mathbb Z$] 
%para classe Z usar ideias pro-p de kz11 \& sz13
Let $d=cd(\Gamma)$ and let
$0\to R_d\to \ldots \to R_0\to \mathbb{Z}\to 0$ be a 
projective resolution %\notetous{free?}
of the trivial ${\mathbb Z}[\Gamma]$-module ${\mathbb Z}$ 
with all $R_i$ finitely generated.
Consider the 
finitely generated profinite projective 
${\mathbb F}_p[\![\widehat{\Gamma}]\!]$-modules $P_i$ defined by 
$P_i=R_i \otimes_{{\mathbb Z}[\Gamma]} {\mathbb F}_p[\![\widehat{\Gamma}]\!]$.
%Note that 
%each $S_i$ is finitely generated profinite projetive ${\mathbb Z}_p[\![\Gamma]\!]$-module.
In light of
\hyperref[l:goody]{Lemma~\ref{l:goody}}, 
the sequence 
$\ldots\to P_2\to P_1 \to P_0 \to \mathbb{F}_p\to 0$
is exact, since its $q$-th homology is isomorphic to
$\varprojlim_N {\rm H}_q(N,\mathbb{F}_p)$
(\textit{cf}. \cite[Lemma~2.1]{KZ:08}).
So $G$ is of type $\operatorname{FP}_\infty$ over $\mathbb{F}_p$.
Finally, since $\mathbb{Z}_p=\varprojlim \mathbb{Z}/p^k\mathbb{Z}$,
by passing to the limit it follows that
$G$ is of type $\operatorname{FP}_\infty$ over $\mathbb{Z}_p$
(\textit{cf}. \cite[Thm.~2.5]{KZ:08}). 

\noindent
(b) Using the notation in item (a), 
consider the 
finitely generated profinite projetive 
${\mathbb Z}_p[\![\widehat{\Gamma}]\!]$-modules $Q_i$ defined by 
$Q_i=R_i \otimes_{{\mathbb Z}[\Gamma]} {\mathbb Z}_p[\![\widehat{\Gamma}]\!]$.
We have that 
$\protect{Q_i \otimes_{ {\mathbb Z}_p[\![\widehat{\Gamma}]\!]} {\mathbb Z}_p}$
is isomorphic to
$R_i \otimes_{ {\mathbb Z}[\Gamma]} {\mathbb Z}_p$
and hence to
$R_i \otimes_{ {\mathbb Z}_p[\Gamma]} {\mathbb Z} \otimes_{\mathbb Z} {\mathbb Z}_p$;
thus its ${\mathbb Z}_p$-rank equals 
${\rm rk}_{{\mathbb Z}}
(R_i\otimes_{{\mathbb Z}[{\Gamma}]} {\mathbb Z})$.
So, the desired Euler-Poincar\'e characteristic over ${\mathbb Z}_p$ can be expressed 
as follows
\[
\aligned
\chi_{{\mathbb Z}_p}(\widehat{\Gamma})
&=\sum_i (-1)^i{\rm rk}_{{\mathbb Z}_p}{\rm H}_i(\widehat{\Gamma},{\mathbb Z}_p) 
=\sum_i (-1)^i{\rm rk}_{{\mathbb Z}_p}
{\rm Tor}_i^{{\mathbb Z}_p[\![\widehat{\Gamma}]\!]}({\mathbb Z}_p,{\mathbb Z}_p) \\
&=\sum_i (-1)^i{\rm rk}_{{\mathbb Z}_p}
(Q_i\otimes_{{\mathbb Z}_p[\![\widehat{\Gamma}]\!]} {\mathbb Z}_p) 
%As in the abstract case
%${\rm rk}_{{\mathbb Z}_p}(S_i\otimes_{{\mathbb Z}_p[\![H]\!]} {\mathbb Z}_p)=
%{\rm dim}_{{\mathbb F}_p}(S_i\otimes_{{\mathbb Z}_p[\![H]\!]} {\mathbb F}_p)$, hence
=\sum_i (-1)^i
{\rm rk}_{{\mathbb Z}}
(R_i\otimes_{{\mathbb Z}[\Gamma]} {\mathbb Z})  \\
&=\sum_i (-1)^i{\rm rk}_{{\mathbb Z}}
{\rm Tor}_i^{{\mathbb Z}[{\Gamma}]}({\mathbb Z},{\mathbb Z})
=\sum_i (-1)^i{\rm rk}_{{\mathbb Z}}{\rm H}_i({\Gamma},{\mathbb Z})
=\chi(\Gamma) \, .
\endaligned
\]

The last assertion of item (b) follows from its discrete version
proved by Kochloukova \cite[Lemma~5]{Kochloukova:10}.
Note 
%finally
that
the cohomological Euler-Poincar\'e characteristic of $\widehat{\Gamma}$ over $\mathbb{F}_p$
equals $\chi_{{\mathbb Z}_p}(\widehat{\Gamma})$
because 
\[
\protect{{\rm dim}_{{\mathbb F}_p}
(S_i\otimes_{{\mathbb F}_p[\![\widehat{\Gamma}]\!]} {\mathbb F}_p)}
=
{\rm rk}_{{\mathbb Z}_p}
(S_i\otimes_{{\mathbb Z}_p[\![\widehat{\Gamma}]\!]} {\mathbb Z}_p)
\, .
\]

\noindent
(c) Fix an index $k$. By Pontryagin duality, 
%${\rm H}_q(U_k,\mathbb{F}_p)$ is isomorphic to ${\rm H}^q(U_k,\mathbb{F}_p^*)^*$ 
${\rm H}_q(U_k,\mathbb{F}_p)\cong {\rm H}^q(U_k,\mathbb{F}_p^*)^*$ 
(\textit{cf.} \cite[Prop.~6.3.6]{RZ:00b}).
Moreover, $U_k$ is of type FP$_{\infty}$ over $\mathbb{Z}_p$, because so is $G$;
hence ${\rm H}^q(U_k,\mathbb{F}_p)$ is finite, and
therefore 
${\rm H}_q(U_k,\mathbb{F}_p)$ is isomorphic to ${\rm H}^q(U_k,\mathbb{F}_p)$.
%${\rm H}_q(U_k,\mathbb{F}_p)\cong {\rm H}^q(U_k,\mathbb{F}_p)$.
Now, %let us prove the second equality. 
%For each index $k$, 
let $N_k$ be the preimage of 
$U_k$ under the canonical map $\Gamma\to \widehat{\Gamma}$.
Note that 
$\dim {\rm H}^q(U_k,\mathbb{F}_p)=\dim {\rm H}_q(N_k,\mathbb{F}_p)$.
For, 
from goodness, 
${\rm H}^q(U_k,\mathbb{F}_p)\cong {\rm H}^q(N_k,\mathbb{F}_p)$,
and by the Universal coefficient theorem,
${\rm H}^q(N_k,\mathbb{F}_p)\cong 
%\textup{Hom}_{\mathbb{F}_p}({\rm H}_q(N_k,\mathbb{F}_p),\mathbb{F}_p)$.
\textup{Hom}({\rm H}_q(N_k,\mathbb{F}_p),\mathbb{F}_p)$.
So, the desired result follows from its discrete version demonstrated by
Bridson and Kochloukova \cite[Cor. B]{BK:17}.
\end{proof}

%%%%%%%%%%%%%%%%%%%%%%%%%%%%%%
%\newpage
%\clearpage 
%\phantomsection
%\section{Groups having no free pro-$p$ subgroup}
%%%%%%%%%%%%%%%%%%%%%%%%%%%%%%

The rest of this section is devoted to the group-theoretic structure of 
%the profinite completion of limit groups.
groups in the class $\mathcal Z$.
To prove
\hyperref[t:virtsol_main]{Theorem~\ref{t:virtsol_main}}
we show next that 
a very strong form of Tits alternative
holds for 
%groups in the class $\mathcal Z$.
these groups.

As generally expected,
the profinite situation is more complex than
discrete or pro-$p$.
For instance, 
this is already evident in the class $\mathcal{Z}_0$:
it is easy to find
a %(pro-finite limit) 
projective profinite group which is
non-free, non-abelian, and virtually $\widehat{\mathbb Z}$
({\it e.g.}, some $\mathbb{Z}_{p'}\rtimes \mathbb{Z}_p$).
We have included %an explicit proof 
a proof
of the following straightforward lemma
for the sake of convenience.

%\newpage
\begin{lemma}%[{\it cf.} \protect{\cite[Ex. 7.7.8]{RZ:00b}}]
\label{l:sylow_cyclic}
Let $G$ be a projective profinite group.
The following assertions are equivalent:
\begin{itemize}
\item[\rm (i)] $G$ has no non-abelian free pro-$p$ closed subgroup,
for any prime $p$;
\item[\rm (ii)] $G$ is meta-procyclic projective
\textup{(}{\it i.e.}, isomorphic to
$\mathbb{Z}_\pi \rtimes\mathbb{Z}_\rho$, 
for some disjoint sets of primes $\pi$ and $\rho$\textup{)};
%\item[\rm (iii)] $G$ is virtually soluble;
%%CANNOT ADD: virtually abelian
%\item[\rm (iv)] $G$ has finite rank. %Pr\"ufer
\end{itemize}
\noindent
Moreover, if $G$ is nilpotent then it is procyclic.
\end{lemma}

\begin{proof} %[Proof \textup{({\it cf}. \protect{\cite[???]{Zalesskii:90}})}.]
%Under each the conditions (i)
%of the lemma,
From condition (i),
every $p$-Sylow subgroup of $G$ is procyclic.
%%Therefore all conditions (i), (iii) and (iv) 
%%hold for $G$, and any two of them are equivalent;
%%and it suffices to reach (ii).
%It suffices to prove that this property implies (ii).
Clearly, the $p$-Sylow subgroups of each finite quotient of $G$
are cyclic.
Hence the commutator subgroup and the abelianization
of each finite quotient of $G$ are
cyclic groups of coprime orders
({\it cf.} \cite[Ch.~V, Thm.~11]{Zassenhaus}).
%%Zassenhaus37, p.139
%({\it cf.} \cite[Proof of 10.1.10]{Rob:96})
Passing to the limit we get that
$\overline{[G,G]}$ and $G/\overline{[G,G]}$
are procyclic groups of coprime orders.
%\cite[Cor. 1.1.8(b)]{RZ:00b} e \cite[Prop. 2.2.4]{RZ:00b}.
Thus, we reach (ii) from the profinite version of Schur-Zassenhaus theorem for
profinite groups
({\it cf.} \cite[Thm. 2.3.15]{RZ:00b}).
%$G$ isomorphic to
%$\mathbb{Z}_\pi \rtimes\mathbb{Z}_\rho$, 
%for some disjoint sets of primes $\pi$ and $\rho$.

Moreover, if $G$ is nilpotent, 
then it is the cartesian product of its Sylow subgroups,
%({\it cf.} \cite[Prop. 2.3.8]{RZ:00b}),
hence procyclic.
%isomorphic to $\mathbb{Z}_{\sigma}$, with $\sigma=\pi \cup\rho$.
\end{proof}

\begin{thm}
\label{t:virtsol}
Let $H$ be a closed subgroup of a group in the class $\mathcal{Z}$.
The following assertions are equivalent:
\begin{itemize}
\item[\rm (i)] $H$ has no non-abelian free pro-$p$ closed subgroup,
for any prime $p$;
\item[\rm (ii)] $H$ is abelian %torsion-free f.g.
or meta-procyclic projective;
%\item[\rm (iii)] $H$ is virtually soluble;
%%CANNOT ADD: virtually abelian
%\item[\rm (iv)] $H$ has finite rank. %Pr\"ufer
\end{itemize}
Moreover, 
nilpotent closed subgroups of a group in the class $\mathcal{Z}$
are abelian, and there exists a uniform \textup{(}finite\textup{)} bound for their rank.
\end{thm}

\begin{proof}
Say $H$ is a closed subgroup of some group $G_n$ of $\mathcal{Z}_n$.
If $n=0$, then 
\hyperref[l:sylow_cyclic]{Lemma~\ref{l:sylow_cyclic}}
concludes the proof.

Suppose then $n\geq 1$.
Let $T$ be the standard tree on which $G_n$ acts,
and restrict this action to $H$.
%Under each one of the conditions (i)-(iv)
%of the theorem,
From \hyperref[l:2acyl]{Lemma~\ref{l:2acyl}},
the only possible cases of
\hyperref[teo:caract-zal90]{Theorem~\ref{teo:caract-zal90}}
%\cite[Thm. 3.1]{Zalesskii:90}, 
are (a) and (c.1).
This completes the result.
\end{proof}

In particular we have proved 
\hyperref[t:virtsol_main]{Theorem~\ref{t:virtsol_main}}.

%\newpage
\begin{cor} \label{c:centre}
%\begin{prop} \label{p:centre}
%Each group in the class $\mathcal Z$? - FALSE!!! Z_p\rtimes Z_q
The profinite completion of a non-abelian limit group
has trivial centre.
%\end{prop}
\end{cor}

\begin{proof}
Let $G$ be the profinite completion of a non-abelian limit group.
If the non-abelian group $G$ is free profinite of finite rank, 
then it has trivial centre $Z(G)$
({\it cf.} \cite[Lemma~8.7.3(a)]{RZ:00b}).

Next, suppose 
$G$ is a closed subgroup of 
%some group $G_n$ of $\mathcal{Z}_n$
%where $G_{n-1}= \widehat{\Gamma_{n-1}}$, 
%$A_{n-1}=\widehat{A_{n-1}}$,
%and $C_{m-1}=\widehat{Z_{n-1}}$
$G_n=
\widehat{\Gamma_{n-1}}\amalg_{\widehat{Z_{n-1}}}\widehat{A_{n-1}}$,
where $\Gamma_{n-1}$, $A_{n-1}$, and $Z_{n-1}$
are as in the proof of \hyperref[t:completion]{Theorem~\ref{t:completion}}.
%Let $T$ be the standard profinite tree on which $G_n$ acts. 
%By Remark \ref{acylindrical tits alternative} either $Z$ fixes a vertex  or it is procyclic acting freely on $T$.
Let $T$ be the standard profinite tree on which $G_n$ acts,
and let $D$ be a minimal $Z(G)$-invariant profinite subtree of $T$
(see \hyperref[prop:invar-zal90]{Proposition~\ref{prop:invar-zal90}}).

If $D$ is a singleton, then 
$Z(G)$ is the stabilizer of a vertex of $T$;
in other words, $Z(G)$ is a subgroup of $g\widehat{\Gamma_{n-1}}g^{-1}$ or 
$g\widehat{A_{n-1}}g^{-1}$
for some $g$ in $G_n$.
In the first situation we use induction on $n$;
in the second, 
since $G$ is non-abelian,
%after a conjugation, 
we obtain %that
$Z(G)=h^{-1}Z(G)h\subseteq \widehat{A_{n-1}}\cap hg\widehat{A_{n-1}}(hg)^{-1}$ 
for any $h$ in $\protect{G\!-\!\widehat{A_{n-1}}}$.
So, by \hyperref[l:2acyl]{Lemma~\ref{l:2acyl}}, we get $Z(G)=\{1\}$.
%(\textit{cf.} \cite[Cor. 2.9]{ZM:88}).

Otherwise, %$D$ would be infinite ({\it cf.} \cite[Prop.~19]{Serre:77}), 
%and 
by \hyperref[l:2acyl]{Lemma~\ref{l:2acyl}},
we could apply \cite[Lemma~2.4(d)]{Zalesskii:90}
to see that $Z(G)$ would act freely on $D$ and be isomorphic to $\mathbb{Z}_{\pi}$ 
for a non-empty set $\pi$ of primes;
moreover, \textit{it would act freely on $T$} (\textit{cf.} \cite[Thm.~2.13]{ZM:88}).
Now, writing $G=\widehat\Gamma$, we may and do assume 
that $\Gamma$ intersects some edge stabilizer of the tree on which it acts; 
for, otherwise $\Gamma$ is a free product,
and hence $G$ is a free profinite product and thus a centreless group
(\textit{cf.} \cite[Thm.~9.1.12]{RZ:00b} or \cite[Thm.~2.13]{ZM:88}).
Thus, the intersection of $G$ with a conjugate of $\widehat{Z_{n-1}}$
is isomorphic to $\widehat{\mathbb{Z}}$.
Therefore, 
the closed subgroup $H$ of $G$ generated by $Z(G)$ and that intersection
would be isomorphic to $\mathbb{Z}_\pi\times \widehat{\mathbb{Z}}$; 
hence $H$ would contain a closed subgroup isomorphic to $\mathbb{Z}_p\times \mathbb{Z}_p$
which, by \hyperref[l:2acyl]{Lemma~\ref{l:2acyl}}
and \hyperref[teo:caract-zal90]{Theorem~\ref{teo:caract-zal90}}, would fix a vertex of $T$.
By \cite[Prop.~2.1]{Zalesskii:90}, 
the minimal $H$-invariant profinite subtree of $T$ would be a vertex.
%contradicting \hyperref[c:virtsol]{Corollary~\ref{c:virtsol}}(i).
So, this case cannot happen.
%Therefore, $Z(G)=\{1\}$
\end{proof}

%%%%%%%%%%%%%%%%%%%%%%%%%%%%%%
%\newpage
%\clearpage 
%\phantomsection
%\section{Finitely generated pro-$p$ subgroups}
%\label{s:fg_pro-p}
%%%%%%%%%%%%%%%%%%%%%%%%%%%%%%

Now we turn to 
\hyperref[t:fg_pro-p_main]{Theorem~\ref{t:fg_pro-p_main}}.
%The next result will be used in 
%\hyperref[s:fg_pro-p]{Section~\ref{s:fg_pro-p}}.
It follows at once from the next more general theorem, which should be compared
with 
\cite[Prop.~2.4]{HZZ:16},
\cite[Thm.~3.5]{SZ:14},
and
\cite[Thm.~11.1]{WZ:17a}.

%Let $n$ be a positive integer.
%A pro-$p$ group is called {$n$-(\emph{free pro-$p$})}
%if each $n$-generated closed subgroup of it is free pro-$p$.
%If a pro-$p$ group is $n$-(free pro-$p$) for each $n$,
%we say that it is $\omega$-(free pro-$p$).

\begin{lemma}[\protect{\textit{cf.} \cite[Lemma~3.7]{WZ:17b}}] \label{l:k_edge_stabilizers}
Let $H$ be a pro-$p$ group acting on a %n infinite 
pro-$p$ tree $T$ 
with compact non-empty edge set %$E(T)$ 
having only finitely many maximal vertex stabilizers up to conjugation. 
Suppose the distance between any two distinct vertices %with non-trivial stabilizers
in the orbits of those finitely many vertices is infinite.
If $k$ is a given integer $\ge 0$,
then $H$ acts on a quotient 
pro-$p$ tree $D$ with non-empty edge set 
such that 
the stabilizer of each edge 
%each edge stabilizer
in $D$ fixes at least $k$ edges in $T$.
\end{lemma}

\begin{proof} 
Let  $H_{v_1}, \ldots, H_{v_n}$ be maximal vertex stabilizers of $T$ up to conjugation; 
each edge or vertex $m$ of $T$ has stabilizer $H_m$ contained in some of $H_{v_i}$ 
up to $H$-translation of $m$. 
Assume $H_m\subseteq H_{v_i}$.
Thus $H_m$ fixes the smallest
pro-$p$ tree $S$ containing $m$ and $v_i$; 
%and, since $E(T)$ is compact,
if $m\neq v_i$, then
each vertex of $S$ %the geodesic 
has an incident edge (\textit{cf.} \cite[Prop.~2.15]{ZM:88}). 
Let $T_k(v)$ be a maximal %connected 
subtree of radius $k$ around~$v$. 
Thus, if $m\neq v_i$, then $m\not\in T_k(v_i)$ and 
the stabilizer $H_m$ fixes at least $k$ distinct edges. 
Now  $\bigcup_{i=1}^n H T_k(v_i)$  is a profinite $H$-subgraph of $T$ 
and collapsing all its connected components to distinct points in $T$ we obtain, 
by \cite[Proposition, p.~486]{Zalesskii:92},
%Proposition on page 486 in \cite{Zal-92} 
a pro-$p$ tree $D$ on which $H$ acts 
with edge stabilizers stabilizing at least $k$ distinct edges of $T$.
\end{proof} 

%\newpage
\begin{thm} \phantomsection \label{t:fg_pro-p}
Let $G$ be a profinite group
acting acylindrically on an infinite profinite tree $T$ with compact edge set 
such that the edge stabilizers are procyclic.
Each finitely generated pro-$p$ closed subgroup $H$ of $G$
splits as
%\[
%P=\left (\amalg^p_{v\in K} (P\cap stab_G(v))\right) \amalg^p F \, ,
%\]
\[
H=\left (H\cap stab_G(v_1)\right) \amalg^p \ldots \amalg^p
\left (H\cap stab_G(v_n)\right)
 \amalg^p F \, ,
\]
where $\amalg^p$ denotes the free pro-$p$ product, 
%$K$ is some vertex subset of $S$, 
$v_1,\ldots, v_n$ are vertices of $T$,
and $F$ is a free pro-$p$ group.
\end{thm}

\begin{proof} By \cite[Thm. 3.5]{SZ:14},
there are only finitely many maximal by inclusion vertex stabilizers in 
$H$ up to conjugation. 
By \cite[Lemma~11.2]{WZ:17a},
we can apply
\hyperref[l:k_edge_stabilizers]{Lemma~\ref{l:k_edge_stabilizers}}
to see that $H$ acts on a profinite tree $D$ such that 
each edge stabilizer stabilizes at least $k$ distinct edges and hence is trivial. 
Then the result follows from \cite[Prop.~2.4]{HZZ:16}.
\end{proof}

This proves \hyperref[t:fg_pro-p_main]{Theorem~\ref{t:fg_pro-p_main}}.
We should mention that, by Lemma \ref{l:2acyl}, the previous theorem follows by direct application of \cite[Thm 3.8]{WZ:17b}.

%%%%%%%%%%%%%%%%%%%%%%%%%%%%%%
%\newpage
%\clearpage 
%\phantomsection
%\section{Commutative-transitive groups}
%%%%%%%%%%%%%%%%%%%%%%%%%%%%%%

%In the present section 
Finally we investigate
the strong property of the transitivity of the commutation.
The conclusion is 
\hyperref[t:CT_main]{Theorem~\ref{t:CT_main}}.
Recall that a group $G$ is 
\emph{commutativity-transitive} 
if for any three non-trivial elements $a$, $b$, and $c$ in $G$
if $a$ commutes with $b$ and $b$ commutes with $c$ then 
$a$ commutes with $c$; 
or equivalently, 
if the centralizer of each non-trivial element in $G$ is abelian.

%\newpage
\begin{prop}  \phantomsection \label{p:centraliz.}
Let $H$ be a closed subgroup of a group in the class $\mathcal{Z}$,
and let $g$ be 
an element of $H$ 
%that generates a free profinite group of rank 1.
such that $\overline{\langle g\rangle}$ is isomorphic to $\widehat{\mathbb Z}$.
%FALSE
%Let $g$ be any non-trivial element of 
%a closed subgroup $G$ of a group in the class $\mathcal{Z}$.
Then:
%We have:
\begin{itemize}
\item[\rm (i)]
the centralizer $C_H(g)$ is maximal abelian
or meta-procyclic projective;
\item[\rm (ii)]
the normalizer $N_H(\overline{\langle g\rangle})$
is maximal abelian or meta-procyclic projective,
it is abelian only if it coincides with $C_G(g)$. %the centralizer.
\end{itemize}
\end{prop}

%\newpage
\begin{proof}  
(i)
It suffices to prove the result for $H=G_n$.
We proceed by induction on $n$.
By simplicity, 
set $G=G_n$ and $C=C_G(g)$.

If $n=0$, 
we have that $\overline{\langle g \rangle}$
is a normal subgroup of the projective profinite group $C$.
Let $S$ be a $p$-Sylow subgroup of $C$.
Then $\overline{\langle g \rangle}\cap S$ is
the $p$-Sylow subgroup of $\overline{\langle g \rangle}$,
%By Lemma \ref{l:free_rk_1}, we have that 
%and thus %%%%%%%%%%%%%%%%%%%\look{elaborate?}
%$S\cong \mathbb{Z}_p$.
and the free pro-$p$ group $S$ has a non-trivial normal procyclic subgroup;
so $S$ is 
isomorphic to $\mathbb{Z}_p$.
%procyclic.
Hence $C$ is meta-procyclic projective, by
\hyperref[l:sylow_cyclic]{Lemma~\ref{l:sylow_cyclic}}.

Suppose then $n\geq 1$.
Let $T$ be the standard profinite tree on which $G_n$ acts. 
By 
\hyperref[prop:invar-zal90]{Proposition~\ref{prop:invar-zal90}}
%Por \cite[Lemma 1.5]{Zalesskii:90}, 
there exists a (non-empty) minimal $C$-invariant profinite subtree $D$. 
When $D$ has only one element, 
there is an element $x$ in $G_n$ such that the conjugate subgroup
$xCx^{-1}$ is contained in $G_{n-1}$ or $A_{n-1}$.
By induction hypothesis, it follows that
$C$ is abelian or meta-procyclic.
Suppose $D$ has more than one element.
Thus, $C$ acts faithfully on $D$ by 
\hyperref[l:2acyl]{Lemma~\ref{l:2acyl}};
hence it is projective by  \cite[Lemma~2.4(d)]{Zalesskii:90}.
The second paragraph of this proof
gives the result.

\medskip

\noindent (ii)
Since the automorphism group
$\operatorname{Aut}(\overline{\langle g\rangle})$ is abelian,
%({\it cf.} \cite[Cor. 4.4.8]{RZ:00b}) logo, 
the group $N_H(\overline{\langle g\rangle})$ is
metabelian-by-abelian,
and thus abelian or meta-procyclic projective, by
\hyperref[t:virtsol]{Theorem~\ref{t:virtsol}}.
The last assertion of the 
theorem
%Proposition 
is obvious.
\end{proof}

Acylindricality shortened the proof but it is not essential to the result.

\begin{cor}
In a group in the class $\mathcal{Z}$,
each closed subgroup having 
$\widehat{\mathbb{Z}}$
%a non-trivial procyclic group
as a direct factor must be abelian.
\hfill \qed
\end{cor}

%\newpage
For the next and last theorem of this section
we recall concepts of covers of profinite groups.
Let $\Phi(G)$ denote the Frattini subgroup
of a profinite group $G$.
We say that an epimorphism of profinite groups
$\varphi\colon \widetilde{G}\to G$
is a {\it projetive} (resp. {\it Frattini}) cover
if, $\widetilde{G}$ is projetive
(resp. $\ker(\varphi)\subseteq \Phi(\widetilde{G})$).
By \cite[Lemma~2.8.15]{RZ:00b}
({\it cf.} \cite[Prop.~22.6.1]{FJ}),
{to each profinite group $G$ there exists an
epimorphism
$\widetilde{\varphi}\colon\widetilde{G}\to G$, 
which is projetive and Frattini.}
Such cover, which is unique up to isomorphism, is called
the {\it universal Frattini cover}.

\begin{thm} \phantomsection \label{p:CT}
A commutativity-transitive group in the class $\mathcal{Z}$
is either abelian or 
%the free pro-$p$ product
%of free pro-$p$ groups or free abelian pro-$p$ groups.
pro-$p$.
\end{thm}

\begin{proof}
A commutativity-transitive group
will be called a CT-group, for simplicity.
Let $G$ be a finitely generated profinite subgroup of some $G_n$
with $G$ a CT-group.
%In the light of 
%\hyperref[t:fg]{Theorem~\ref{t:fg_pro-p}}
%it suffices to prove that $G$ is either abelian or pro-$p$.
The proof is by induction on $n$.

If $n=0$, then $G$ is either procyclic or non-abelian free pro-$p$.
Indeed, %suppose $G$ is non-abelian.
%If each finite quotient of $G$ is a $p$-group, 
%then $G$ is free pro-$p$.
%Otherwise, 
let 
$Q$ be a non-abelian finite quotient of $G$
having elements of distinct coprime orders. %$p$ e $q$.
Then, the canonical map $G\to Q$ factors through
the universal Frattini cover
$\widetilde{\varphi}\colon\widetilde{Q}\to Q$
%(aqui $\widetilde{Q}$ \'e projetivo, 
%$\widetilde{\varphi}$ \'e epimorfismo, 
%${\rm ker}(\widetilde{\varphi})\subseteq \Phi(\widetilde{Q})$
%e $d(\widetilde{Q})=d(Q)$
%- vide \cite[Prop. 22.6.1]{FJ}
%ou ainda \cite[Lemma 2.8.15]{RZ:00b}),
(since $G$ is projective),
in such a way that $G\to \widetilde{Q}$ is also an epimorphism
(because $\ker(\widetilde{\varphi})\subseteq \Phi(\widetilde{Q})$).
%aqui usar (1a) implica condicao (1b):
%subgp de \widetilde{Q} coincide com \widetilde{Q} sss \widetilde{\varphi}(subgp)=Q
%tomar subgp = (imagem de G em \widetilde{Q})
Now, 
the Frattini subgroup
$\Phi(\widetilde{Q})$ is an open normal subgroup of $\widetilde{Q}$, 
and $\widetilde{Q}$ has a non-trivial $p$-Sylow, say $S_p$.
Intersecting $S_p$ with $\Phi(\widetilde{Q})$
we obtain a $p$-Sylow subgroup of $\Phi(\widetilde{Q})$,
which has finite index in $S_p$.
Thus, since $\Phi(\widetilde{Q})$ is pronilpotent,
%({\it cf.} \cite[Prop. 2.5.1(d)]{Wilson}),
Sylow subgroups of $\widetilde{Q}$
relative to distinct primes must commute pointwise;
hence $\widetilde{Q}$ is pronilpotent.
%(vide \cite[Prop. 2.4.3]{Wilson}).
On the other hand, being $\widetilde{Q}$ projetive, 
it can be seen as a subgroup of $G$; 
so $\widetilde{Q}$ is also a CT-group.
Therefore,
%Sendo $\widetilde{Q}$ n\~ao-pro-$p$
each Sylow subgroup of $\widetilde{Q}$ is abelian;
hence so is $\widetilde{Q}$.
%A contradiction to the definition of $\widetilde{Q}$.

%Suppose %that 
Next, assume
the result is true for the class $\mathcal{Z}_{n-1}$.
%(p.~\pageref{defi:gl}).
Consider $G_n\in\mathcal{Z}_n$ where
$G_n=G_{n-1}\amalg_{C_{n-1}} A_{n-1}$
with $A_{n-1}=B_{n-1}\times C_{n-1}$ %;
and $G_{n-1}\ne C_{n-1}\ne A_{n-1}$.
Let $D_{n-1}$ be
a non-trivial procyclic maximal direct factor of $B_{n-1}$
(which may coincide with $B_{n-1}$) and
define
$K=\overline{\langle gD_{n-1}g^{-1}\,|\,g\in G_n\rangle}$.
By \cite[Thm.~B]{Zalesskii:95},
$K$ is a projective profinite group.
\textit{Consider $H=K\cap G$.}
%\textit{Let $H$ be the intersection of $K$ and $G$.}
Since in the second paragraph of the proof
we did not use the fact that $G$ is finitely generated,
we have three possibilities for $H$:
non-abelian free pro-$p$, non-trivial procyclic, or trivial.
To conclude the proof
we henceforth suppose that $G$ is not pro-$p$.

First, we claim that $H$ cannot be a non-abelian free pro-$p$ group.
Indeed, 
since there exists an element $y$ in $G$ such that
$\overline{\langle y\rangle}$ is isomorphic to $\mathbb{Z}_q$
with $q\ne p$, and since $G$ is a CT-group, %ent\~ao 
%pelo Corol\'ario \ref{cor:direto}
the subgroup $\overline{\langle y\rangle}$ would act non-trivially
on the normal subgroup $H$;
that is, there would exist $x$ in $H$ that would not commute with $y$.
So the subgroup $\overline{\langle x,y\rangle}$
would be CT, projetive, non-abelian and non-pro-$p$;
a contradiction to the second paragraph.

Second, suppose that $H$ is non-trivial procyclic.
We take
a non-trivial $\ell$-Sylow subgroup $S$ of $H$,
which is characteristic in $H$, 
and hence normal in $G$.
If $S$ is central in $G$, then $G$ must be abelian.
Otherwise, there would exist an element in $G$
that would not commute with $S$ and 
hence there would exist a subgroup $T$,
of order a power of a prime $r$,
%isomorfo a $\mathbb{Z}_r$,
that would not commute with $S$.
Consider then $S.T$, which would be isomorphic to 
$\mathbb{Z}_\ell\rtimes\mathbb{Z}_r$.
If $r\ne \ell$, then this product, and hence $G$, would not be CT
(because a non-abelian CT-group is centreless);
the case $r=\ell$, could not happen by Theorem \ref{t:virtsol}.

Third, suppose $H$ is trivial.
Note that, $K$ is the kernel of the epimorphism
$G_n\longrightarrow G_{n-1}\amalg_{C_{n-1}} (A_{n-1}/D_{n-1})$
induced by the identity map of $G_{n-1}$
and by the canonical map $A_{n-1}\to A_{n-1}/D_{n-1}$.
%Then, $G$  is isomorphic to $\pi(G)$ and, 
%either $A_{n-1}/D_{n-1}$ is isomorphic to $C_{n-1}$
%or %escolhemos $D_{n-1}\cong\widehat{\mathbb{Z}}$ de modo que
%$d(A_{n-1}/D_{n-1})<d(A_{n-1})$.
%The result follows
%by induction on %$n$ and 
%$d(A_{n-1})$.
Thus 
%$G$ embeds in $G_{n-1}\amalg_{C_{n-1}} (A_{n-1}/D_{n-1})$.
we have an embedding
\[
G \longrightarrow G_{n-1}\amalg_{C_{n-1}} (A_{n-1}/D_{n-1}) \, .
\]
If $d(A_{n-1})=1$, then $A_{n-1}/D_{n-1}$ is $C_{n-1}$
and we are done;
otherwise, $d(A_{n-1}/D_{n-1})<d(A_{n-1})$.
So, the result follows 
inductively. 
%from this.
\end{proof}

\begin{cor}
The profinite completion of a limit group $\Gamma$
is commutativity-transitive
if, and only if, 
$\Gamma$ is a free abelian group of finite rank.
\end{cor}

\begin{proof}
Being $\Gamma$ residually free, 
if it is non-abelian,
then it has a finite non-$p$-group as a quotient;
hence its profinite completion is not a pro-$p$ group.
\end{proof}

%%%%%%%%%%%%%%%%%%%%%%%%%%%%%%
%\newpage
\phantomsection
\section{Final comments} %Some other properties
%%%%%%%%%%%%%%%%%%%%%%%%%%%%%%

In this last section we make a few comparisons 
concerning group-theoretical 
properties of
limit groups, ``limit pro-$p$ groups'' (\textit{i.e.,} those studied in \cite{KZ:11}),
and groups in the class $\mathcal{Z}$.

First, 
if $p$ and $q$ are distinct primes
and $q$ divides $\max\{2,p-1\}$,
%we 
consider ${\mathbb{Z}_p\rtimes\mathbb{Z}_q}$.
This example shows that
non-abelian groups in the class $\mathcal{Z}$ may have non-trivial centre
(versus non-abelian limit groups must have trivial centre);
moreover,
virtually abelian groups in the class $\mathcal{Z}$ are not necessarily abelian
(versus virtually nilpotent limit groups must be abelian).

Next,
groups in $\mathcal{Z}$ may have finite continuous abelianization;
\textit{e.g.}, the universal Frattini cover of a 
non-abelian finite simple group
has trivial continuous abelianization
(because 
any Frattini cover of a perfect profinite group is perfect).
%Seja $\pi\colon X\to Y$
%um recobrimento de Frattini qualquer de um grupo profinito $Y$.
%Ent\~ao, um subgrupo fechado $X_0$ de $X$ \'e igual a $X$
%se, e somente se, $\pi(X_0)=Y$. %[FJ, Definition 22.5.1]
%Logo, se $Y$ \'e um grupo profinito perfeito
%({\it i.e.} $Y=\overline{[Y,Y]}$),
%ent\~ao $X$ tamb\'em o \'e.

Furthermore, 
%projective groups have non-negative deficiency 
%({\it cf.} \cite[Cor.~1.2]{Lubotzky:01}), for instance 
the non-procyclic group $\widehat{\mathbb{Z}}\amalg \mathbb{Z}_p$
has deficiency one, while 
every closed abelian subgroup of it is at most 1-generated.
For limit groups or ``limit pro-$p$ groups'', 
if every closed abelian subgroup at most 1-generated
but the groups themselves are not 1-generated
then their deficiency must be $\ge 2$.
%%as was shown by Lubotzky [24], Corollary 1.2, all projective groups (i.e. proﬁnite groups of co-homological dimension 1) have non-negative deﬁciency as well.
%So, even though non-1-generated groups with the property that 
%every abelian subgroup is at most 1-generated 
%do not necessarily have deficiency at least two,
%it makes sense to ask if these groups have 
%exponential subgroup growth.

Also, 
%({\it i.e.} the intersection of any two finitely generated closed subgroups 
%is not necessarily finitely generated)
Howson's property fails for groups in the class $\mathcal{Z}$
already at $n=0$:

\begin{prop} %\textup{(}A. Jaikin-Zapirain \cite{J:15}\textup{)}
The free profinite group of rank 2
does not satisfy
the finitely generated intersection property of Howson.
\end{prop}

\begin{proof} \textup{(A. Jaikin-Zapirain \cite{J:15})}.
Let $p$ and $q$ be two different primes. 
Pick a closed normal subgroup $N$ of a free profinite group $F$ of rank 2
such that $F/N$ is isomorphic to the free pro-$q$ group $A$ with basis $\{a\}$.
%$F/N\cong \mathbb{Z}_q$ with $\{a\}$ being a generator.
Let $H=N/N_p$ be the maximal pro-$p$ quotient of $N$.
So, the 2-generated group $G=F/N_p$ is 
a semi-direct product of the infinitely generated free pro-$p$ group $H$ and $A$.

Now, let $T$ be a free pro-$q$ group with basis $\{g_1,g_2\}$. 
Define its action on $H$ in such a way that $g_1$ and $g_2$ act as does $a$. 
Form a semidirect product $T\ltimes H$ and, for $i=1,2$, take  
the subgroup $G_i$ topologically generated by $g_i$ and $H$ therein.
Being $G_1$ and $G_2$ isomorphic to $G$, they are 2-generated.

So, $T\ltimes H$ is a 3-generated profinite group. 
Since its Sylow subgroups are free, it has cohomological dimension 1, 
and therefore it is a subgroup of a free profinite group. 
But $G_1\cap G_2=H$ is not finitely generated.
\end{proof}

%Let pp and qq be two different primes. First I want to construct two generated profinite group G=A⋉HG=A⋉H isomorphic to semi direct product of a infinitely generated free pro-pp group HH and a cyclic free pro-qq group A=⟨a⟩A=⟨a⟩.
%
%It can be done, for example, in the following way. Start with a two generated free profinite group FF and a normal subgroup NN such that F/N≅AF/N≅A. Let H=N/NpH=N/Np be the maximal pro-p quotient of NN. Clearly HH is infinitely generated. Then G=F/Np≅A⋉HG=F/Np≅A⋉H satisfies the requiered conditions.
%
%Now, let T=<g1,g2>T=<g1,g2> be a two generated free pro-qq group. Define its action on HH in such way that g1g1 and g2g2 act as it does aa. Form a semidirect product T⋉HT⋉H and put Gi=<gi,H>Gi=<gi,H> (i=1,2i=1,2). It is clear that Gi≅GGi≅G are two generated.
%
%T⋉HT⋉H is a 3-generated profinite group. Since its Sylow subgroups are free, T⋉HT⋉H has cohomological dimension 1, and so it is a subgroup of a free profinite group. However G1∩G2=HG1∩G2=H is not finitely generated.

%If $G=\mathbb{F}_p^{3}$, then
%$d(\widetilde{G})=d(G)=3$, where
%$\widetilde{G}$ denotes 
%Moreover, the universal Frattini cover of an elementary abelian group of rank 3
%is generated by 3 elements and no less than 3 elements.
%Thus an interesting question is whether
%every $2$-generated profinite subgroup of a group 
%in the class $\mathcal Z$ is abelian or projective.

%Maximal abelian non-procyclic closed subgroup of the free profinite of rank 2?
%holds for free profinite because the such set is empty
%the problem is induction!

Finally, 
a non-abelian free profinite group
cannot be residually soluble.
%({\it cf.} \cite[Sec.~3.4]{RZ:00b}).
%Remark: %on the other hand, 
We do not know whether
the ``limit pro-$p$ groups" are necessarily residually nilpotent.

%%%%%%%%%%%%%%%%%%%%%%%%%%%%%%%
%%\newpage
%\phantomsection
%\section{Open questions}
%%%%%%%%%%%%%%%%%%%%%%%%%%%%%%%
%
%At the moment the following 
%natural but challenging questions are known to be affirmative
%for discrete limit groups and are also open for their 
%pro-$p$ analogues.
%Which of the following properties does any group $G$ in the class $\mathcal{Z}$ have?
%
%
%1. residually-(free profinite)
%({\it i.e.} $G$ is the subcartesian product of free profinite groups);
%
%2. %Does $G$ satisfies 
%the finitely generated intersection property of Howson
%({\it i.e.} the intersection of any two finitely generated closed subgroups is itself 
%finitely generated);
%%F=completion of L(x,y),
%%K=ker (cl.[F,F] \to maximal prosol. quot of F) is not f.g.!
%%and K=K<x> \cap K<y>
%
%3. locally a virtual retract
%({\it i.e.}, for each finitely generated closed subgroup $H$ of $G$,
%does there exist an open subgroup $V_H$ of $G$ containing $H$
%such that the inclusion $H\hookrightarrow V_H$ has a left inverse)?
%
%%REMOVER 2 e 3 se valerem para PROJETIVOS E VIRTUALMENTE

%\end{document}

%%%%%%%%%%%%%%%%%%%%%%%%%%%%%%
%\newpage
\phantomsection
\section*{Acknowledgements}
\label{s:acknowl}
%\addcontentsline{toc}{section}{Acknowledgements}
%%%%%%%%%%%%%%%%%%%%%%%%%%%%%%

Both authors are grateful for the
partial financial support from CNPq and CAPES.

%%%%%%%%%%%%%%%%%%%%%%%%%%%%%%
%\newpage

\end{document}